\documentclass[12pt]{amsart}

\usepackage{xcolor,hyperref}
\usepackage{comment}
\definecolor{red}{rgb}{1.00,0.00,0.00}
 
\numberwithin{equation}{section}

\newtheorem{theorem}{Theorem}[section]

\newtheorem{corollary}[theorem]{Corollary}

\newtheorem{remark}{Remark}
\newtheorem{definition}[theorem]{Definition}

\def\sA{\langle A\rangle}
\def\sB{\langle B\rangle}
\def\sC{\langle C\rangle}

\def\C{\mathfrak{C}}
\def\N{\mathbb{N}}
\def\Q{\mathbb{Q}}

\def\a{{\bf a}}
\def\b{{\bf b}}

\newcommand{\bs}{\backslash}

\newcommand{\la}{\langle}
\newcommand{\ra}{\rangle}

\definecolor{purple(x11)}{rgb}{0.63, 0.36, 0.94}

\def\sA{\langle A\rangle}
\def\sB{\langle B\rangle}
\def\sC{\langle C\rangle}

\title{Wilf inequality is preserved under Gluing of Semigroups}
\author{Srishti Singh and Hema Srinivasan}
\address{University of Missouri, Columbia, MO 65211, USA}
\date{March 2023}
 
\thanks{H. Srinivasan is supported by grants from Simons Foundation. S. Singh is supported by Dissertation Year Fellowship from Missouri.}
\begin{document}

\maketitle
 
\begin{abstract} Wilf Conjecture on numerical semigroups is a question posed by H.Wilf in 1978 and is an inequality connecting the Frobenius number, embedding dimension and the genus of the semigroup.  The conjecture is still open in general.  We prove that this Wilf inequality is preserved under gluing of numerical semigroups.  If the numerical semigroups minimally generated by $A = \{ a_1, \ldots, a_p\}$ and $B = \{ b_1, \ldots, b_q\}$ satisfy the Wilf inequality, then so does their gluing which is minimally generated by $C =k_1A\sqcup k_2B$. We discuss the extended Wilf's Conjecture in higher dimensions under the process of gluing.
\end{abstract}

\section{Introduction}
A numerical semigroup $\la A\ra$ is a submonoid of $\N$  minimally generated by  $A= \{a_1, \ldots a_p\}$ where  $\gcd (a_1, \ldots, a_p)= 1$.  
Such a semigroup will contain all but finitely many positive integers.  The largest positive integer not in the numerical semigroup $\la A\ra$ is called the Frobenius number $F_A$.  
We denote by $[n_A]$ the set of $n_A$ elements in $\la A \ra$ that are less than $F_A$. Thus, the number of gaps is $F_A- n_A$. Moreover, $n_A \ge 1$ as $0\in \la A\ra$ for every numerical semigroup $A$.  

One of the most intriguing questions in the realm of numerical semigroups is the Wilf Conjecture, named after H. Wilf, posed in 1978. This conjecture establishes an inequality that relates three fundamental invariants of a numerical semigroup: the minimal number of generators (or the embedding dimension), the Frobenius number, and the number of gaps. To be precise, the Wilf inequality for a numerical semigroup $A$, minimally generated by $p$ elements, is $pn_A \ge F_A+1$. Despite significant efforts and a multitude of research, the Wilf Conjecture remains unsolved.

This note addresses an important construction in the realm of numerical semigroups: gluing. We demonstrate that gluing preserves the Wilf property in numerical semigroups, offering a means to generate more semigroups satisfying the conjecture. The preservation of the Wilf property under gluing is a matter of significant interest, as it allows for its extension to higher dimensions. In Section $\ref{affinecase}$, we discuss the ramifications on the Extended Wilf Conjecture under gluing of semigroups in higher dimensions.  The concept of gluing in numerical semigroups arose in the successful attempts to characterize the complete intersection numerical semigroups \cite{De} and \cite{ros}.  Gluing is known to preserve singularity classes such as Cohen-Macaulay and Gorenstein (\cite{GS22}, Theorem $1.5$). It is especially of interest in the numerical case, as any two such groups can be glued.

When $\sA$ is symmetric, $F_A+1 = 2n_A$, hence, Wilf Conjecture is true for all symmetric semigroups, i.e. for all Gorenstein semigroups or semigroup rings of type $1$.  Further, it is known from \cite{Fr}, that for any numerical semigroup of type $t$, $(t+1) n_A \ge F_A+1$.  Thus, Wilf Conjecture is true whenever the type is less than the embedding dimension. In particular, all numerical semigroups generated by arithmetic sequence must satisfy Wilf inequality as their type is less than the embedding dimension (\cite{A}, Prop. $20$).  However, the type of a numerical semigroup can be any positive number when the embedding dimension is four or more. For an example, see \cite{Fr}. This integration of gluing and the notable Wilf conjecture highlights the impact of gluing as a construction method for creating new numerical semigroups that retain the desirable Wilf property.

\section{Gluing Preserves Wilf's Inequality }
\begin{definition}
Let $A = \{a_1, \ldots, a_p\}$ and $B= \{b_1, \ldots, b_q\}$ minimally generate two numerical semigroups $\sA$ and $\sB$. A numerical semigroup $\sC$ is said to be the gluing of $\sA$ and $\sB$ if the set of minimal generators of $C$ is a disjoint union $k_1 (a_1, \ldots, a_p) \sqcup k_2(b_1, \ldots, b_q)$ where $k_1\in \sB$ but not in $B$, and $k_2 \in \sA$ but not in $A$. 
\end{definition}
\begin{remark}
    Since $\sC$ is a numerical semigroup, $k_1$ and $k_2$ must be relatively prime.
\end{remark}

In the domain of numerical semigroups, Delorme initially introduced this type of decomposition (\cite{De}) to characterize complete intersections. Delorme demonstrated that a numerical semigroup is a complete intersection if and only if it is a gluing of two complete intersections and more impressively, the converse, that is all complete intersections with an embedding dimension of at least three must be a gluing of two other complete intersections.

Later, \cite{ros} Rosales developed a criterion for gluing,  initially in the cntext of numerical semigroups. This characterization extends to higher dimensions. This notion has since been extensively studied in \cite{GS1}, and \cite{GS2}. Notably, various invariants of the glued semigroup $\sC$ can be efficiently computed from the invariants of the constituent semigroups $\sA$ and $\sB$. 

We denote a gluing of $\sA$ and $\sB$ as $C = k_1 A \Join k_2 B$. The formula for the Frobenius number $F_C$ in the glued semigroup $\sC$ can be derived using the regularity formula of the graded Betti numbers of $k[A]$ and $k[B]$, as demonstrated in \cite{GS1}. However, in this note, we provide a direct and straightforward proof.
\begin{theorem}\label{Frob}
Let $A = \{a_1, \ldots, a_p\}$ and $B= \{b_1, \ldots, b_q\}$ minimally generate two numerical semigroups $\sA$ and $\sB$ and let  $C = k_1A\sqcup k_2B$ so that $\sC = \sA \Join \sB$. 
Then $$F_C = k_1F_A+k_2F_B+k_1k_2$$
\end{theorem}

\begin{proof}
Firstly, since $k_1, k_2$ are relatively prime, then for any positive integer $t\ge 1$, we can find positive integers $m$ and $n$, with $0< n < k_1$ such that $mk_1-nk_2 = t$. Then $ k_1F_A+k_2F_B+k_1k_2+t = k_1(F_A+m)+k_2(F_B -n+k_1) \in \sC$. We will now show that $ k_1F_A+k_2F_B+k_1k_2 \not\in \sC$. 

Suppose $k_1F_A +k_2F_B+k_1k_2 = k_1\sum_{i=1}^pr_ia_i +k_2 \sum_{j=1}^q s_jb_j$ for some $r_i, s_j \in \N$, $i \in \{1,...,p\}, j \in \{1,...,q\}$. Then we have $k_1( F_A-\sum_{i=1}^pr_ia_i) = k_2(\sum_{j=1}^qs_jb_j- k_1-F_B)$. 

Since $k_1 $ and $k_2$ are relatively prime and $F_A\not\in \sA$, we must have $F_A-\sum_{i=1}^p r_ia_i= k_2t$ for some non zero integer $t$. But then $F_A = \sum_{i=1}^pr_ia_i +tk_2 \in \sA$ if $t>0$.  So, $t<0$. Hence, $F_B= -tk_1-k_1+\sum_{j=1}^q s_j b_j = \sum_{j=1}^q s_jb_j+(-1-t) k_1\in \sB$, which is not possible either. So, $k_1F_A+k_2F_B+k_1k_2 \not\in \sC$. Thus, $F_C =  k_1F_A+k_2F_B+k_1k_2$.
\end{proof}

The following theorem establishes that if two numerical semigroups $A$ and $B$ satisfy the Wilf Conjecture, then so does every gluing of $A$ and $B$. 
\begin{theorem}\label{Gluingiswilf}
Suppose $\sA$ and $\sB$ satisfy the Wilf inequality. Then so does  $\sC = \sA\Join \sB$.
\end{theorem}

\begin{proof}
Suppose $A = \{a_1, \ldots, a_p\}$ and $B= \{b_1, \ldots, b_q\}$ minimally generate two numerical semigroups $\sA$ and $\sB$, and let  $C = k_1A\sqcup k_2B$ so that $\sC = \sA\Join \sB$. 

Thus, $k_2 \in \la A \ra$, not in $A$, and  $k_1\in \la B \ra$ and not in $B$, and $(k_1,k_2) =1$. By Theorem $\ref{Frob}$, $F_C = k_1F_A+k_2F_B+k_1k_2$.  Without loss of generality, we may take $p\le q$.  

To begin with, we have a theorem of Froberg et al (\cite{Fr}) which states $(t+1)n_C \ge F_C+1$ for all numerical semigroups of type $t$. By \cite{GS1}, the type of $\sC$ is the product of types of $\sA$ and $\sB$. When the embedding dimension of $\sA$ or $\sA$ is 1 or 2, the type of $\sC$ is 1 as it is a complete intersection. Thus, if $p\le 2$ and $q\le 2$, then  $C$ is also a complete intersection and $F_C< 2n_C \le (p+q) n_C$ as required.  

If $p\le q = 3$, then again since $t\le 2$, we get that $$F_C < (t+1)n_C \le 3n_C \le (p+q)n_C.$$ Thus, we just need to consider the case when $q\ge 4$.

Now, for all $x\in [nA], 1\le i\le k_1$, $k_1x+ k_2(F_B+i) \in [n_C]$. Further, $$k_1x+ k_2(F_B+i) = k_1x'+k_2(F_B+j) \implies k_1(x-x') = k_2(j-i).$$ But $k_1, k_2$ are relatively prime, and $j-i\le k_1$. So $j=i,x=x'$. Hence, all these aforementioned elements are distinct in $[n_C]$.

Similarly, for $1\le i\le \lfloor k_2/2\rfloor, 1\le j\le \lfloor k_1/2\rfloor$,
the elements $k_1(F_A+i)+k_2(F_B+j) \in [n_C]$, and all these $\lfloor k_2/2\rfloor \lfloor k_1/2\rfloor$ elements are distinct.

So, $$n_C \ge k_1n_A+ \lfloor k_2/2\rfloor \lfloor k_1/2\rfloor \ge k_1n_A+ (k_1-1)(k_2-1)/4$$

Similarly, $$n_C \ge k_2n_B+ \lfloor k_2/2\rfloor \lfloor k_1/2\rfloor \ge k_2n_B+ (k_1-1)(k_2-1)/4$$

Now suppose $q\ge p\ge 2$, so that $F_A $ and $F_B$ are both positive integers. Hence, 
\begin{multline*}
    F_C= k_1F_A+ k_2(F_B+k_1)= k_1(F_A+1) +k_2(F_B+1) + k_1k_2-k_1-k_2 \\
    = k_1(F_A+1)+k_2(F_B+1) + (k_1-1)(k_2-1)-1
\end{multline*}
Since $A$ and $B$ satisfy the Wilf inequality, $F_A +1\le pn_A$ and $F_B+1\le qn_B$. It now follows that 
\begin{align*}
    F_C = &k_1(F_A+1)+k_2(F_B+1) + (k_1-1)(k_2-1)-1  \\ 
    \le &k_1pn_A+k_2qn_B+ (k_1-1)(k_2-1)-1
\end{align*}
Next, $${\frac {F_C}{n_C} } = {\frac {k_1(F_A+1)}{n_C} }+{\frac {k_2(F_B+1) + (k_1-1)(k_2-1)-1}{n_C}}$$ Recall that $n_C\ge k_1n_A$, and $n_C \ge k_2n_B+ (k_1-1)(k_2-1)/4$. Hence,  $${\frac {F_C}{n_C} } \le {\frac {k_1pn_A}{k_1n_A} }+{\frac {k_2qn_B + (k_1-1)(k_2-1)-1}{k_2n_B+ (k_1-1)(k_2-1)/4}}$$ But $${\frac {k_2qn_B + (k_1-1)(k_2-1)-1}{k_2n_B+ {\frac {(k_1-1)(k_2-1)}{4}} }}\le q {\frac {k_2n_B + {\frac {(k_1-1)(k_2-1)-1}{q}}}{k_2n_B+ {\frac{(k_1-1)(k_2-1)}{4}}} }< q$$ since $q\ge 4$.  

Thus, $F_C <( p+q)n_C$, as long as $p\ge 2$. It remains to show the inequality if $p=1$ and $q\ge 4$. In that case, 
\begin{align*}
(q+1)n_C - F_C  \ge  &(q+1) (k_2n_B + (k_1-1)(k_2-1)/4) - k_2F_B- k_1(k_2-1) \\
 = & k_2(qn_B-F_B-1) + (k_1k_2-k_2-k_1+1)\left(\frac{q+1}{4} - 1\right) + k_2n_B -1 \\
> & 0,
\end{align*}
since $qn_B \ge F_B+1$ , ${\frac {q+1}{4}} \ge 1$ and $k_2n_B -1\ge 2-1>0$.

So, we have that Wilf Conjecture holds for $\sC$ as well. 
\end{proof}
\renewcommand{\a}{\bf a}
\renewcommand{\b}{\bf b}

\section{Gluing in Higher Dimensions}\label{affinecase}
We now want to explore the same question in higher dimension, namely for subsemigroups of $\N^n$ as opposed to in $\N$. A finitely generated submonoid of $\N^n$ is also called an \textit{affine semigroup}. Thus, if $A $ is an $n\times p$ matrix in $\N$ with columns $\mathbf a_i, 1\le i\le n$, representing elements of $\N^n$, the monoid generated by the columns of $A$ in $\N^n$ is an affine semigroup $\sA$.  As before without loss of generality, we may assume that the entries in $A$ are relatively prime as $\sA$ and $\langle dA\rangle$ are isomorphic as monoids.    The question then becomes, given two numerical semigroups $\sA$ and $\sB$ in $\N^n$, $n \geq 2$, is there a generalized notion of this notable conjecture by Wilf, and, if so, does a gluing of $\sA$ and $\sB$ satisfy this notion if $\sA$ and $\sB$ do?

The concept of gluing of numerical semigroups in higher dimensions has been elucidated in \cite{GS2} and \cite{GS1}. Consider finite subsets $A=\{\mathbf a_1, ..., \mathbf a_p\}$, $B=\{\mathbf b_1,...,\mathbf b_q\}$ and $C= k_1A \sqcup k_2B$ in $\N^n$ where $n \geq 2$, and $k_1$ and $k_2$ are positive integers. Let $\sA, \sB, \sC$ be the semigroups generated by the aforementioned sets, and $k[A] \cong k[\mathbf x_1,...,\mathbf x_p]/ I_A, k[B] \cong k[\mathbf y_1,...,\mathbf y_q]/ I_B$, and $k[C] \cong k[\mathbf x_1,...,\mathbf x_p,\mathbf y_1,...,\mathbf y_q]/ I_C$ be their respective semigroup rings. Set $R:= k[\mathbf x_1,...,\mathbf x_p,\mathbf y_1,...,\mathbf y_q]$. 
\begin{definition}
$\sC$ is said to be a gluing of $\sA$ and $\sB,$ written $C=A \Join B$, if $I_C = I_AR + I_BR + \la \rho \ra$ with $\rho = \bf x^{\bf c} - \bf y^{\bf d}$ for some $\bf c \in \N^p$ and $\bf d \in \N^q$. 
\end{definition}

Although any two numerical semigroups can be glued, it is not true in higher dimensions.  An affine semigroup $\sA \in \N^n$ is said to be non-degenerate if $A$ spans $\Q^n$.  Two non-degenerate semigroups in dimension $\ge 2$ cannot be glued. Precise conditions when such gluings exist can be found in \cite{GS22}.

 To extend the Wilf inequality to higher dimensions, it is necessary to have a sense of finiteness attached to the semigroup, so as to have an analogue of Frobenius number and the number of gaps.  There is a notion of \textit{Generalised Numerical Semigroup} (\cite{CDFFP}), namely, if $H(S) := \N^n \bs S$ is finite, we call $S$ a \textit{Generalised Numerical Semigroup} (GNS). Further, if $\mathcal{H}(S) := (\text{cone}(S)\bs S) \cap \N^n$ is finite, where cone$(S)$ is the cone of $S$ in $\Q^n_{\geq 0}$, then we call $S$ a $\mathcal C$-\textit{semigroup}. Thus, $S$ is a $\mathcal C$ semigroup if and only if there are only finitely many elements in $\N^n$ such that no integral multiples of them are in $S$.  

Suppose $S$ is a $\mathcal C$-semigroup (or GNS). Let $\prec$ be a monomial ordering satisfying the following condition:
\begin{equation}\label{ns}
    \mathbf a \in \N^n \implies |\{\mathbf b \in \N^n \mid \mathbf b \prec \mathbf{a}\}| < \infty
\end{equation}
Let $F(S)$, called the \textit{Frobenius Element} of $S$, denote the maximal element of the set $\mathcal H(S)$ (respectively, $H(S)$) with respect to $\prec$.  Set $\mathcal{N}(S) = |\mathcal H(S)| + |\{\mathbf x \in S \mid \mathbf x \prec F(S)\}|$ (respectively, $\mathcal{N}(S) = |H(S)| + |\{\mathbf x \in S \mid \mathbf x \prec F(S)\}|$). This exists due to the restriction applied on $\prec$. The Extended Wilf Conjecture is  formulated in \cite{wilfext} as follows $$|\{ \mathbf x \in S \mid \mathbf x \prec F(S) \}| \cdot \: e(S) \geq \mathcal{N}(S) + 1$$
\renewcommand{\C}{\mathcal C}

\begin{remark}
    
If $\sA$ and $\sB$ are generalized numerical semigroups, and $\sC = \sA \Join  \sB$, then $\sC$ cannot be a GNS. 
\end{remark}
To see this, we will first show that if $S=\langle \mathbf s_1,...,\mathbf s_t \rangle$ is a GNS in $\N^n$, then $\text{rank } S = n$, where $S$ also represents the $n \times t$ matrix $\begin{bmatrix}
    \mathbf{s}_1 & \cdots & \mathbf s_t
\end{bmatrix}$, that is,  a GNS must be non-degenerate. For any $\mathbf x \in \N^n$, if the rank of $S$ is strictly less than $n$, then one of the rows $R_i, 1\le i\le n$ of the matrix $S$,  say the $nth$, can be written as $\sum_{i=1}^{n-1}r_i R_i$  for some $r_i\in \Q$.   Then any element $(t_1, \ldots, t_n)^T\in \N^n$ must satisfy $t_n = \sum_{i=1}^{n-1}r_i t_i$.  Thus, there are infinitely many elements in $\N^n$ that will be missing from $S$.  This implies that $\N^n \bs S$ is not finite, contradicting the fact that $S$ is a GNS. 

It now follows from \cite{GS1} Corollary $1.6$, which states that if $\sA$ and $\sB$ can be glued, then $\text{rank} A + \text{rank} B = n+1$, where $A$ is the $n \times p$ matrix $
\begin{bmatrix}
\mathbf a_1 & \cdots & \mathbf a_p
\end{bmatrix}$, and $B$ is the $n \times q$ matrix $
\begin{bmatrix}
    \mathbf b_1 & \cdots & \mathbf b_q
\end{bmatrix}
$, that two generalized numerical semigroups $\sA$ and $\sB$ cannot be glued in $\N^n$ to obtain another GNS if $n \geq 2$. 

Thus, to generalize theorem \ref {Gluingiswilf}, we will consider gluings of $\mathcal C$-semigroups. Unlike Generalised Numerical Semigroups, whose gluings are never a GNS themselves, gluing of $\C$-semigroups can be a $\C$-semigroup. This is the only case when we can study the extended Wilf Conjecture under the gluing process, as we show in the following result. We thank Carmelo Cisto and Om Prakash for pointing out an error in the computations of gluing higher dimensional $\C$-semigroups in the earlier version.
\begin{theorem}
    A necessary condition for a gluing of $\C$-semigroups to be a $\C$-semigroup is $k_1=1$.
\end{theorem}
\begin{proof}
   Let $\sC = k_1 \sA \Join k_2 \sB$ be a gluing of $\C$-semigroups $\sA$ and $\sB$. Suppose $k_1>1$. Recall that $A= \{\mathbf a_1,..., \mathbf a_p\}$, and $A$ will also denote $n \times p$ matrix with columns $\displaystyle{\mathbf a_i = \begin{pmatrix}
        a_{i1}\\
        \vdots \\
        a_{in}
    \end{pmatrix}}, 1 \leq i \leq p$. Let rank $A=r \leq n$, and denote by $\mathbf u$ the gluable lattice point of $\sA$ and $\sB$. Due to the rank conditions on $A$ and $B$, we may assume that $\text{cone}(A)= \text{ span}_{\Q_{\geq 0}} \left\{\mathbf a_1',..., \mathbf a_{r-1}', \mathbf u \right\}$, where $\mathbf a_i = \gcd(a_{i1},...,a_{1n}) \mathbf a_i'$. Denote by $F$ the face span$_{\Q_{\geq 0}}\{\mathbf a_1'\}$ of cone$(A)$.

    By Theorem 9 of \cite{DRGGMAVT}, $\sC$ is a $\C$-semigroup only if $F \cap$ cone$(C) \bs \sC$ is finite, and by Lemma 9 (\cite{DRGGMAVT}),  $F \cap$ cone$(C) \bs \sC$ is finite if and only if $\gcd(\{\lambda \mid \lambda \mathbf a_1' \in F \cap C \})=1$. Suppose $\{a_{i_1},...,a_{i_l}\} \subseteq A \cap F$. Note that necessarily $l \geq 1$ and $a_1 \in A \cap F$. Then $\{k_1a_{i_1},...,k_1a_{i_l}\} \subseteq C \cap F$, and the aforementioned gcd is at least $k_1>1$.

Thus, $\mathcal H(C)$ is infinite, and $\sC$ cannot be a $\C$-semigroup.
\end{proof}
\begin{remark}

 Note that when $k_2 >1,$ and rank $B \geq 2$, then a symmetric argument shows that $\sC$ cannot be a $\C$-semigroup. This condition is not sufficient.  One can have $k_1=k_2=1$ for two $\C$- semigroups with the gluing not be a $\C$- semigroup. When $k_1=1$ and rank $B =1$, then for any $k_2$, $\sC$ is a $\C$-semigroup.  If $\sC$ is a gluing, then $B = \mathbf b[u_1,u_2, \ldots, u_q]$ for some $\mathbf b \in \sA, u_1,...,u_q \in \N$ (Theorem $1.7$, \cite{GS22}), and hence,  $\sA=\sC$. Thus, $\sC$ trivially satisfies the Extended Wilf Conjecture.     
\end{remark}
\begin{corollary}
Let $\sC = \sA \Join \sB$ be a gluing of two semigroups $\sA$ and $\sB$, where rank of $B$ is $1$ and $\sA$ and $\sB$ satisfy the extended Wilf Conjecture.  Suppose that the gluing $\sC$ is a $\C$-semigroup. Then $k_1=1$ and the extended Wilf conjecture is satisfied. 
\end{corollary}

Our findings present an intriguing observation regarding Wilf's inequality. While seemingly sharp in $\mathbb{N}$, the inequality takes on a different nature in the extended setting of $\mathbb{N}^n,$.  

Based on these compelling observations, we put forth
 a conjecture: 
 Let $C = A \Join B$, where $\sA$, $\sB$, and $\sC$ are all $\C$ semigroups. If $\sA$ and $\sB$ satisfy the extended Wilf inequality, then $\sC$ also inherits the Wilf inequality.

\end{document}